\documentclass{amsart}

\usepackage{amssymb}

\newtheorem{theorem}{Theorem}[section]
\newtheorem{corollary}{Corollary}[section]
\newtheorem{lemma}[theorem]{Lemma}

\theoremstyle{definition}

\theoremstyle{remark}
\newtheorem{remark}[theorem]{Remark}

\numberwithin{equation}{section}
\usepackage{todonotes}
\usepackage{xcolor}
\usepackage{pstricks}
\usepackage{graphicx}

\begin{document}


\title{Knotted toroidal sets, attractors and incompressible surfaces}


%    author one information
\author{H\'ector Barge}
\address{E.T.S. Ingenieros Inform\'aticos, Universidad Polit\'ecnica de Madrid, 28660 Madrid, Espa\~na}
\curraddr{}
\email{h.barge@upm.es}
\thanks{The authors are partially supported by the Spanish Ministerio de Ciencia e Innovaci\'on (grants PGC2018-098321-B-I00 and PID2021-126124NB-I00)}

%    author two information
\author{J.J. S\'anchez-Gabites}
\address{Facultad de Ciencias Matem\'aticas, Universidad Complutense de Madrid, 28040 Madrid, Espa\~na}
\curraddr{}
\email{jajsanch@ucm.es}
\thanks{The second author is funded by the Ram\'on y Cajal programme (RYC2018-025843-I)}

%    \subjclass is required.
\subjclass[2020]{Primary 57K99, 37B35, 37E99; Secondary 55N99}
\keywords{Attractor, toroidal set, incompressible surface}
\date{}

%    Abstract is required.
\begin{abstract}
In this paper we give a complete characterization of those knotted toroidal sets that can be realized as attractors for both discrete and continuous dynamical systems globally defined in $\mathbb{R}^3$. We also see that the techniques used to solve this problem can be used to give sufficient conditions to ensure that a wide class of subcompacta of $\mathbb{R}^3$ that are attractors for homeomorphisms must also be attractors for flows. In addition we study certain attractor-repeller decompositions of $\mathbb{S}^3$ which arise naturally when considering toroidal sets.  
\end{abstract}

\maketitle

\section{Introduction}

Throughout this paper we consider dynamical systems which are either discrete (generated by iteration of a homeomorphism) or continuous (given by a continuous flow; no smoothness assumptions are made). We shall focus on attractors; i.e. compact invariant sets which are stable in the sense of Lyapunov and attract nearby points.

It is well known that attractors for dynamical systems might exhibit a very complicated topological structure, and one may wonder whether this topological complexity is essentially arbitrary. More formally, one poses the following realization problem: characterize topologically which compacta $K$ in some given phase space can be realized as attractors for a dynamical system. Actually this encompasses a broad family of problems depending on what sort of dynamics and phase space one considers; one could even replace attractors by more general invariant sets. Thus in the literature one finds realization problems of attractors for topological flows (\cite{garay1}, \cite{gunthersegal1}, \cite{mio5}, \cite{sanjurjo1}), for analytic flows (\cite{jimenezllibre1}, \cite{peraltajimenez1}), critical sets rather than attractors (\cite{graysonpugh1}, \cite{nortonpugh1}, \cite{Souto}), attractors for homeomorphisms (\cite{gunther1}, \cite{kato1}, \cite{ortegayo1}, \cite{mio6}, \cite{mio7}), or attractors for iterated function systems instead of single homeomorphisms (\cite{crovisierrams1}, \cite{duvallhusch1}).

The realization problem is relatively easy to solve in $\mathbb{R}^2$: among all compacta, attractors (either for homeomorphisms or flows) are characterized by having a finitely generated \v{C}ech cohomology. The situation in $\mathbb{R}^3$ is much more complicated. Unlike the $2$--dimensional case, now whether a compactum $K \subseteq \mathbb{R}^3$ can be realized as an attractor is related not only to topological properties of $K$ itself (such as its cohomology, as before), but crucially also to the way it sits in $\mathbb{R}^3$. This is illustrated by the fact that the round $3$--ball can trivially be realized as an attractor but it can also be embedded in a different way (for instance, as Alexander's horned ball) which cannot be realized as an attractor. The same phenomenon occurs with many other simple sets, such as arcs or spheres (\cite{mio6}).

In $\mathbb{R}^3$ (unlike $\mathbb{R}^2$, again) being an attractor for a flow and for a homeomorphism are no longer equivalent; for instance the dyadic solenoid is an attractor for a homeomorphism but not for a flow. A topological characterization of compacta in $\mathbb{R}^3$ which can be realized as attractors for flows is known (\cite{gunthersegal1}, \cite{mio5}). For discrete dynamics the realization problem seems to be very difficult and is wide open, although partial non-realizability results are known. In order to make some progress it is then sensible to focus on particular classes of compacta $K$. 

A compact set $K \subseteq \mathbb{R}^3$ is called cellular if it has a neighbourhood basis comprised of $3$--cells (i.e. handlebodies of genus zero). The dynamical properties of cellular sets (in arbitrary dimensions) were studied by Garay \cite{garay1}, who showed that all of them can be realized as attractors of flows. Toroidal sets were introduced in \cite{hecyo1} as a natural next step in complexity. A compactum $K \subseteq \mathbb{R}^3$ is called toroidal if it is not cellular and has a neighbourhood basis comprised of solid tori (i.e. handlebodies of genus one). For example, solenoids, generalized solenoids, knotted solenoids, tame knots and some wild knots constructed by summing infinitely many knots are toroidal sets. From a dynamical perspective, any attractor which has a trapping region that is a solid torus will be either a cellular or a toroidal set.

Unlike cellular sets, toroidal sets can exhibit ``knotting'' and ``winding'' because so can solid tori. In fact it is possible to associate to any toroidal set $K$ two topological notions, its genus $g(K)$ (a nonnegative integer which is a generalization of the classical genus of a knot) and its self-index $\mathcal{N}(K)$ (an Abelian group), that formalize these intuitive ideas. A toroidal attractor has a finite genus and also a certain finiteness property in $\mathcal{N}(K)$, and so the genus and the self-index provide obstructions for the realizability of a toroidal set as an attractor. For example, certain infinite connected sums of knots which are wild just at a single point, or certain generalized solenoids, or any knotted solenoid, are all examples of toroidal sets which cannot be realized as attractors of homeomorphisms of $\mathbb{R}^3$ (\cite{hecyo1} and \cite{hecyo2}).

In Section \ref{sec:toroidal} we obtain a complete solution to the realization problem for knotted toroidal sets. A toroidal set is knotted if it has a positive genus, or equivalently every sufficiently small torus neighbourhood of the toroidal set $K$ is knotted (in the sense that its core curve is a nontrivial knot). In Theorem \ref{lem:unknotted} we shall show that a knotted toroidal set can be realized as an attractor for a homeomorphism if and only if it can be realized as an attractor for a flow; Theorem \ref{thm:charact} then provides explicit characterizations of when this realization as an attractor is possible. These theorems improve earlier results from \cite{hecyo2} by removing assumptions about the cohomology of $K$. A  proof of Theorem \ref{lem:unknotted} obtained by different and more complicated means can be seen in \cite{hecyo3}.

Although somewhat obscured by technical details, the basic idea in the proofs is quite simple. The key ingredient is a finiteness result of Kneser and Haken which essentially states that for every compact $3$--manifold $M$ there exists a number $n=n(M)$ such that any collection $\{S_1,\ldots,S_m\}$ of $m > n(M)$ disjoint incompressible surfaces in $M$ contains a parallel pair, i.e. some pair $S_i, S_j$ in the collection cobounds a region of $M$ which is homeomorphic to a product $S \times [0,1]$. (The precise statement, along with the necessary background, is given in Section \ref{sec:background}). If one has an attractor $K$ for a homeomorphism $f$ and picks a compact neighbourhood $N$ of $K$ contained in $\mathcal{A}(K)$, iterating $\partial N$ backwards with a suitable power of the dynamics produces a countable family of disjoint surfaces contained in the compact manifold $M:= \mathbb{S}^3 \backslash {\rm int}\ N$. If one can ensure that these are incompressible (in the case of knotted toroidal sets this is guaranteed by taking $N$ to be a knotted solid torus), and regardless of the value of $n(M)$, at least two of them must be parallel. One can then use these to produce a flow which has $K$ as an attractor, thus falling back onto the case of flows which is fully understood as mentioned above.

The argument outlined in the previous paragraph not only works for knotted toroidal sets but in general. This is the content of Section \ref{sec:generalize}. The reason for discussing first the particular case of knotted toroidal sets is that it is easier to visualize and also the incompressibility assumption needed in Theorem \ref{thm:nbhinc} is automatic. Similar techniques involving compressibility arguments were used in \cite{Jiang1} in order to characterize those closed orientable $3$-manifolds admitting a diffeomorphism possessing a  knotted Smale solenoid contained in its non-wandering set.

The paper is organized as follows. Section \ref{sec:background} contains some background definitions and results. Most of these are needed to sort out the various technical issues that appear in carrying out the programme described earlier. Section \ref{sec:toroidal} solves the realization problem for knotted toroidal sets; it also lays out the essential ideas for the proofs in the general case discussed in Section \ref{sec:generalize}. Section \ref{sec:att-rep} discusses certain attractor-repeller decompositions of $\mathbb{S}^3$ which arise naturally when considering toroidal sets.

\section{Background} \label{sec:background}

\subsection{Attractors} \label{subsec:dynamics} We  consider $\mathbb{R}^3$ as our phase space. The dynamics will be invertible; i.e. either a (complete) flow or a discrete dynamical system generated by the iteration of a homeomorphism $f$ of $\mathbb{R}^3$. Sometimes it will be convenient to work in $\mathbb{S}^3 = \mathbb{R}^3 \cup \{\infty\}$ and then we shall assume $f$ to be extended to a homeomorphism of $\mathbb{S}^3$ by letting $\infty$ be fixed.

We formulate the following well known definitions in the case of discrete dynamics; their translation to flows is straightfoward. An attractor $K$ for $f$ is a compact invariant ($f(K) = K$) set which is stable in the sense of Lyapunov and attracts nearby points, that is, it has an open neighbourhood $U$ such that for every point $p \in U$ the forward iterates $f^n(p)$ eventually enter any prescribed neighbourhood of $K$. The largest $U$ with this property is called the basin of attraction of $K$ and denoted by $\mathcal{A}(K)$. It is an open, invariant neighbourhood of $K$. The stability condition on $K$ implies that $K$ attracts not only points but in fact compacta in $\mathcal{A}(K)$: for every compactum $P \subseteq \mathcal{A}(K)$ and every neighbourhood $V$ of $K$ there exists $n_0$ such that $f^n(P) \subseteq V$ for every $n \geq n_0$.

We will often use the following manipulation, which works only in the context of discrete dynamics. If $K$ is an attractor for $f$ it is easy to check that $K$ is also an attractor, and with the same basin of attraction, for any power $f^r$ of $K$ (with $r \geq 1$). This has the following useful consequence: given any compact neighbourhood $N \subseteq \mathcal{A}(K)$, after replacing the dynamics with a suitable power we may assume that $N$ is positively invariant. Indeed, since $K$ attracts compacta in $\mathcal{A}(K)$ we may find $r_0$ such that $f^r(N) \subseteq {\rm int}\ N$ for every $r \geq r_0$; then replacing $f$ with $f^{r_0}$ turns $N$ into a positively invariant set without altering the fact that $K$ is an attractor with basin of attraction $\mathcal{A}(K)$. (This is generally not possible for flows; in fact, the topological type of positively invariant neighbourhoods for flows is essentially determined by $K$).

In a locally connected and locally compact space an attractor has finitely many connected components $K_1, \ldots, K_r$, each contained in a component of the basin of attraction. In the case of flows each $K_i$ is invariant and an attractor whose basin of attraction is the component of $\mathcal{A}(K)$ that contains it. In the case of discrete dynamics this is not necessarily true since the components of $K$ might get permuted among themselves, but after replacing $f$ with $f^{r!}$ all components will be invariant and each will be, as in the case of flows, an attractor whose basin of attraction is the component of $\mathcal{A}(K)$ that contains it. This allows us to work always with connected attractors without loss of generality.

Suppose $K$ is an attractor for a flow. Let $P \subseteq \mathcal{A}(K)$ be a compact positively invariant neighbourhood of $K$. Using the flow to construct suitable homotopies it can be shown that both inclusions $K \subseteq P \subseteq \mathcal{A}(K)$ induce isomorphisms in \v{C}ech cohomology (see \cite{gunthersegal}, \cite{hastings1}, \cite{kaprod}, \cite{sanjurjo4}, 
\cite{sanjurjo}). For discrete dynamics the situation is more complicated since one has to replace the flow-induced homotopies by algebraic arguments, but one can show (see for instance \cite{gobbino1} and \cite{pacoyo1}) that the inclusion $K \subseteq \mathcal{A}(K)$ induces isomorphisms in \v{C}ech cohomology with coefficients in a field and also with coefficients in $\mathbb{Z}$ provided any one of $\check{H}^*(K;\mathbb{Z})$ or $\check{H}^*(\mathcal{A}(K);\mathbb{Z})$ is finitely generated. We will concentrate on attractors $K$ that are connected and nonseparating, which amounts to saying $\check{H}^0(K;\mathbb{Z}_2) = \mathbb{Z}_2$ and (by Alexander duality) $\check{H}^2(K;\mathbb{Z}_2) = 0$, so we will always have $H^2(\mathcal{A}(K);\mathbb{Z}_2) = 0$. By the universal coefficient theorem $H_2(\mathcal{A}(K);\mathbb{Z}_2) = 0$, which intuitively means that $\mathcal{A}(K)$ contains no cavities. Notice that in the case of discrete dynamics, and unlike the case of flows, the inclusion $K \subseteq P$ does not necessarily induce isomorphisms in cohomology. This is directly related to the fact discussed earlier that one can achieve that any prescribed $N$ become positively invariant, without changing $K$ or $\mathcal{A}(K)$, by taking a suitable power of the dynamics.

Attractors for flows have several special properties. By using a Lyapunov function one can construct compact positively invariant neighbourhoods $P \subseteq \mathcal{A}(K)$ of the attractor whose boundary is transverse to the flow. This ensures that $(\partial P) \times (-1,1)$ is homeomorphic (via the flow) to an open neighbourhood of $\partial P$ in $\mathbb{R}^3$. Using results from the theory of homology manifolds this entails that $\partial P$ is a $2$--manifold (see \cite[Corollary 4, p. 74]{chewningowen1}) and so $P$ is a compact $3$--manifold. Since the homology of $P$ is determined by that of $K$ by the results mentioned in the previous paragraph, this sometimes allows to identify $P$ completely. For example if $K$ has the \v{C}ech cohomology of a point then the same is true of $P$; Lefschetz duality shows that $\partial P$ has the homology of a sphere and finally the Sch\"onflies theorem for bicollared spheres implies that $P$ is a $3$--ball. As a consequence $K$ is cellular, since flowing $P$ forwards with the dynamics produces a neighbourhood basis of $K$ comprised of $3$--cells.

Another feature of attractors for flows is that they sit in phase space in a relatively tame way which we describe now (details can be found in \cite{gunthersegal} and \cite{mio5}). Again by picking a compact positively invariant neighbourhood $P \subseteq \mathcal{A}(K)$ whose boundary is transverse to the flow we have that $P \setminus K$ is homeomorphic (via the flow) to $(\partial P) \times [0,+\infty)$. For brevity we say that $P \setminus K$ has a product structure. Conversely, if $K$ is a compact set and $P$ is a compact neighbourhood of $K$ such that $P \setminus K$ has a product structure (where one of the factors is $[0,+\infty)$, as before), then one can easily construct a flow on the ambient manifold having $K$ as an  attractor as follows. Endow $[0,+\infty)$ with the flow $(s,t) \longmapsto s e^t$. This has $0$ as a fixed point which repells everything towards $+\infty$. Extending it to a flow on $(\partial P) \times [0,+\infty)$ which leaves the first coordinate unchanged one obtains a flow that pushes everything towards infinity uniformly. Using the homeomorphism $P \setminus  K \cong (\partial P) \times [0,+\infty)$ this flow can be copied to $P \setminus K$. Then one can slow down the trajectories as they approach $K$ in such a way that they become stationary ``at time $+\infty$''. This slowed down flow can then be extended by the identity to $K$ and realizes it as an attractor. The ``slowing down'' construction is trivial in the differentiable context (just multiply the velocity field by some suitable damping function) but not so much in the topological setup. It has been described in great generality by Beck \cite[Chapters 4 and 5]{beck1} but in our case the flow in ${\rm int}\ P \setminus K$ is parallelizable and then the elementary argument in the proof of \cite[Lemma 3.2]{hecyo3} suffices.

\subsection{Attractors and piecewise linearity} \label{subsec:plapprox} The tools of $3$--manifold topology work best in a piecewise linear context, and on occasion we will need our dynamical objects to be in this category. We discuss two instances of this.

The stability condition on attractors ensures that they have neighbourhood bases comprised of positively invariant sets; i.e. compact neighbourhoods $P \subseteq \mathcal{A}(K)$ with the property that $f(P) \subseteq {\rm int}\ P$. It is easy to see that these can always be taken to be polyhedral $3$--manifolds, as follows. First one covers $P$ with finitely many tiny cubes to obtain a slightly larger $P'$ which is now a polyhedron (it being a union of cubes) that still satisfies $f(P') \subseteq {\rm int}\ P$ (provided the cubes are taken tiny enough) and hence $f(P') \subseteq {\rm int}\ P'$. Then we take a regular neighbourhood of $N$ (again, sufficiently small) to turn it into a polyhedral $3$--manifold which is still positively invariant.

If $f$ is a homeomorphism having an attractor $K$, it is always possible to replace $f$ with a homeomorphism $f'$ which still has $K$ as an attractor and is piecewise linear on $\mathbb{R}^3 \setminus K$ (or $\mathbb{S}^3\setminus K$), perhaps at the cost of changing the basin of attraction of $K$. This relies on classical approximation theorems of $3$--manifold topology and is proved in  \cite[Lemma 7.5, p. 73]{hecyo2}. The proof given there in fact establishes a slightly stronger result which will be needed in Sections \ref{sec:att-rep} and \ref{sec:generalize} and controls to what extent the basin of attraction of $K$ is altered by this manipulation: given any positively invariant neighbourhood $P$ of $K$ in $\mathcal{A}(K)$, one can take $f'$ in such a way that $P$ is still contained in the basin of attraction of $K$ for the new map $f'$ and is, in fact, positively invariant for it.

\subsection{Toroidal sets} \label{subsec:toroidal}

As mentioned in the Introduction, a toroidal set is a compactum $K \subseteq \mathbb{R}^3$ that is not cellular and has a neighbourhood basis of solid tori $\{T_n\}$; i.e., sets homeomorphic to $\mathbb{S}^1 \times \mathbb{D}^2$. We will always assume the $T_n$ to be nested in that $T_{n+1} \subseteq {\rm int}\ T_n$ for every $n$. Since solid tori are connected and do not separate $\mathbb{R}^3$, the same is true of toroidal sets. It follows from a deep approximation result of Moise \cite[Theorem 1, p. 253]{moise2} that one can always perturb each $T_n$ by an arbitrarily small amount to turn it into a polyhedral torus. Thus if necessary we can take neighbourhood bases $\{T_n\}$ where each $T_n$ is polyhedral, and this will be done tacitly.

Toroidal sets $K$ may be classified according to their one dimensional integral \v Cech cohomology group $\check{H}^1(K;\mathbb{Z})$. If $\{T_n\}$ is any (nested) neighbourhood basis for $K$, the continuity property of  \v Cech cohomology entails that  $\check{H}^1(K;\mathbb{Z})$ is the direct limit of the following direct limit:
\[
\check{H}^1(T_1;\mathbb{Z})\longrightarrow\check{H}^1(T_2;\mathbb{Z})\longrightarrow\cdots
\]
where the arrows are induced by the inclusions $T_{n+1}\hookrightarrow T_{n}$. Since each $T_n$ is a solid torus, each of these arrows can be identified with $\mathbb{Z}\xrightarrow{\cdot w_n}\mathbb{Z}$. Geometrically, $w_n$ is the algebraic number of times that $T_{n+1}$ winds inside $T_n$. It is then easy to see (see \cite[Proposition~1.6]{hecyo1}) that there are three mutually exclusive possibilites for $\check{H}^1(K;\mathbb{Z})$: \label{eq:coh}
\begin{enumerate}
    \item[(1)] $\check{H}^1(K;\mathbb{Z})=0$. This happens whenever infinitely many of the $w_n$ vanish. In this case we say that $K$ is a \emph{trivial} toroidal set.
    \item[(2)] $\check{H}^1(K;\mathbb{Z})=\mathbb{Z}$. This is the case when $w_n=1$ for all but finitely many $n$.
    \item[(3)] Otherwise $\check{H}^1(K;\mathbb{Z})$ is not finitely generated.
\end{enumerate}

The Whitehead continuum (see p. \pageref{fig:white}) is an example of a trivial toroidal set. Any tame knot is an example of a toroidal with $\check{H}^1=\mathbb{Z}$, as is any toroidal attractor for a flow (as argued below). Any generalized solenoid is an example of a toroidal set of the third kind.

The amount of knottedness of a toroidal set $K$ can be measured by its genus. This quantity, introduced in \cite{hecyo1} and denoted by $g(K)$, is defined as the minimum $g\in\{0,1,\ldots,\infty\}$ such that $K$ possesses arbitrarily small neighbourhoods that are polyhedral solid tori of genus $\leq g$. The genus of a solid torus is defined to be the genus of a core curve in the ordinary sense of knot theory. (This definition is correct because the core of a solid torus is unique up to isotopy). A toroidal set has genus zero if and only if it has a neighbourhood basis of (polyhedral) unknotted solid tori. In that case we also say that $K$ itself is unknotted. If $K$ has a positive genus we say that it is knotted. Notice that a toroidal set may have an infinite genus.

\subsection{Toroidal attractors}\label{subsec:toratt} Suppose $K$ is a toroidal set which is an attractor for a homeomorphism $f$. Since $\mathcal{A}(K)$ is a neighbourhood of $K$ there exists a solid torus $T \subseteq \mathcal{A}(K)$ which is a neighbourhood of $K$, and by replacing $f$ with some iterate we may take $T$ to be positively invariant. As mentioned earlier we may assume $T$ to be polyhedral and $f$ to be piecewise linear outside $K$, and then $\{T_n\} := f^n(T)$ is a neighbourhood basis of $K$ which consists of polyhedral solid tori. Since all the $T_n$ are ambient homeomorphic they are all knotted in the same way, so they all have the same genus which is then the genus of $K$. Thus in particular the genus of a toroidal attractor is finite. Also, by construction each $T_{n+1}$ sits inside $T_n$ in the same way that $T_1$ sits inside $T_0$ so all the winding numbers $w_n$ are equal to some $w$ and so $\check{H}^1(K;\mathbb{Z})$ is either zero (if $w = 0$), $\mathbb{Z}$ (if $w=1$) or the $w$--adic integers if $w \geq 1$.

Now suppose $K$ is a toroidal attractor for a flow. Then there is a solid torus neighbourhood $T$ of $K$ contained in $\mathcal{A}(K)$. Recalling that the inclusion $K \subseteq \mathcal{A}(K)$ induces isomorphisms in cohomology and factoring it through $T$ one see that $\check{H}^1(K;\mathbb{Z}) = 0$ or $\mathbb{Z}$. The first case was discussed earlier and led to the conclusion that $K$ is cellular, which is ruled out by the definition of toroidal sets. Thus we must have $\check{H}^1(K;\mathbb{Z}) = \mathbb{Z}$.

Two nested solid tori $T' \subseteq {\rm int}\ T$ are called concentric if $T  \setminus  {\rm int}\ T'$ is homeomorphic to the product of a $2$--torus and a closed interval. Let $K$ be a toroidal set which is an attractor for a homeomorphism $f$ and let $T_0$ be a solid torus neighbourhood for $K$ contained in its basin of attraction. Suppose that there exists an iterate $f^r$ (with $r \geq 1$) such that $f^r(T_0)$ is contained in the interior of $T_0$ and is concentric with it. Then $T_n := f^{nr}(T_0)$ is a neighbourhood basis of $K$ that consists of concentric solid tori, and so $T_0  \setminus  K$ is homeomorphic to a $2$--torus times a half--open interval. This product structure allows one to realize $K$ as an attractor for a flow as described in Subsection \ref{subsec:dynamics}. The converse is also true: if a toroidal set is an attractor for a flow, then it has a neighbourhood basis of concentric solid tori. Essentially the reason is that being an attractor for a flow forces any positively invariant neighbourhood $P$ to be a solid torus (this is a nontrivial fact), and then iterating $P$ forwards with the time-one map of the flow produces a neighbourhood basis of $K$ comprised of concentric tori $P_n$. They are concentric because the flow itself provides a homeomorphism between $(\partial P) \times [n,n+1]$ and the region $P_n \backslash {\rm int}\ P_{n+1}$. Details can be found in \cite[Theorem~3.11]{hecyo1}.

% A direct consequence of this result together with the continuity property of \v Cech cohomology is that if a toroidal $K\subset\mathbb{R}^3$ is an attractor for a flow then $\check{H}^1(K;\mathbb{Z})\cong\mathbb{Z}$ \todo{Esto se puede usar para decir que el Whitehead no es atractor de un flujo como comentas más arriba}. 

\subsection{Some $3$--manifold topology} \label{subsec:3mfds}

This section collects some definitions and results pertaining to $3$--manifold topology, taylored to our needs. We follow the classical books by Hempel \cite{hempel1} and Jaco \cite{jaco1}. Every object is assumed to be polyhedral.

We consider a $3$--manifold $M \subseteq \mathbb{R}^3$, possibly with a nonempty compact boundary $\partial M$. By a surface $S$ in $M$ we mean a compact $2$--manifold without boundary. We always suppose that $S$ is either contained in $\partial M$ (and it is therefore a union of components of $\partial M$) or entirely contained in ${\rm int}\ M$. A compressing disk for $S$ (in $M$) is a disk $D \subseteq M$ that intersects $S$ precisely along $\partial D$ and such that $\partial D$ is not nullhomotopic in $S$. Observe that $2$--spheres do not have compressing disks.

The surface $S$ is said to be compressible in $M$ if it has a compressing disk $D$ or it is a $2$--sphere which bounds a $3$--ball in $M$. Otherwise, $S$ is called incompressible. A compressible surface $S$ having a compressing disk $D$ can be cut along it: one removes a thin open annulus along $\partial D$ from $S$ and caps the (newly created) two boundary components of the resulting surface with two parallel copies of $D$. This produces a new surface $S' \subseteq M$ which is still compact and boundariless and is simpler than $S$ in that the nonzero element of $\pi_1(S)$ represented by $\partial D$ has now been killed. It is straightforward to check that $\chi(S') = \chi(S) + 2$, and this implies that iterating this process of cutting along compressing disks will always lead in a finite number of steps to a (generally not connected) surface in $M$ each of whose components has no further compressing disks (perhaps because it is a $2$--sphere). When $S$ is a component of $\partial M$ the whole manifold $M$ can be cut along the compressing disk to obtain a new manifold $M'$. The resulting surface $S'$ is then still contained in $\partial M'$.

To illustrate the definition, and for later reference, we consider some examples. (See \cite[III.2 and 3]{jaco1} for more details).

(1) Take $M$ to be a solid torus and $S$ its boundary. A compressing disk for $S$ in $M$ is then just a meridional disk of the torus. Notice that cutting the torus $M$ along the compressing disk produces a ball contained in the torus.

(2) As a variation of (1) consider now a nested pair of solid tori $T' \subseteq T$. Let $M := T \setminus {\rm int}\ T'$ and take $S = \partial T$ to be the boundary of the bigger torus. A compressing disk $D$ for $S$ in $M$ is now a meridional disk of $T$ that misses $T'$ (so that it is contained in $M$). Such a disk $D$ not always exists; in fact, it exists if and only if there is a ball $B$ such that $T' \subseteq B \subseteq T$. This is very easy to prove. If a compressing disk $D$ exists then we may cut $T$ along it to obtain a ball $B$ that contains $T'$ (because $D$ was disjoint from $T'$). Conversely, if such a ball $B$ exists then by performing a homotethy inside it we may shrink $T'$ to a tiny torus in $T$ and in particular find a meridional disk of $T$ that misses $T'$. Such a meridional disk is a compressing disk for $S$ in $M$.

(3) Another standard example is that of $M := \mathbb{S}^3 \setminus {\rm int}\ T$ where $T$ is a solid torus. The boundary $\partial M$ is compressible in $M$ if and only if $T$ is unknotted.

The following result records the crucial fact that a compact manifold $M$ only admits finitely many ``essentially different'' incompressible surfaces. This is known as the Kneser--Haken finiteness theorem. We use the formulation given in \cite[III.20, p. 42]{jaco1} removing the condition of $\partial$--incompressibility since our surfaces are always taken to be boundaryless:

\begin{theorem} \label{thm:kneserhaken} Let $M$ be compact orientable $3$--manifold. There exists an integer $n$ which depends only on $M$ such that any collection of pairwise disjoint, closed, connected incompressible surfaces $S_1, \ldots, S_k$ with $k > n$ contains a pair of parallel surfaces.
\end{theorem}

Two surfaces $S_i$ and $S_j$ are parallel if they cobound a connected component of $M \backslash \bigcup S_i$ that is homeomorphic to the product  $(\partial S_i) \times [0,1]$.

Finally, we shall make use of a very useful algebraic characterization of compressibility. Suppose $S$ is a connected component of $\partial M$ or a connected, $2$--sided surface in the interior of $M$. If $S$ is compressible in $M$ (and is not a $2$--sphere) then the boundary of a compressing disk is a nonzero element in the kernel of the inclusion induced homomorphism $\pi_1(S) \longrightarrow \pi_1(M)$. Due to the loop theorem of Dehn and Papakyriakopoulos the converse is also true and leads to the following characterization (\cite[Lemma 6.1, p. 58]{hempel1}): a component $S$ of $\partial M$ or a connected, $2$--sided surface $S \subseteq {\rm int}\ M$ other than a $2$--sphere is incompressible in $M$ if and only if the inclusion induced homomorphism $\pi_1(S) \longrightarrow \pi_1(M)$ is injective.

\subsection{Algebraic background} \label{subsec:prelim_algebra}

We will make use of the notion of an amalgamated product. We recall now some basic facts about it (see \cite[Section 4.2]{magnuskarrasssolitar}).

Suppose $A, B$ are groups and $C$ is another group which injects into $A$ and $B$ through homomorphisms $\alpha : C \longrightarrow A$ and $\beta : C \longrightarrow B$. Then one can form the amalgamated product $A *_C B$, which is essentially the free product of $A$ and $B$ modulo the identification $\alpha(c) = \beta(c)$ for every $c \in C$ (or, more precisely, modulo the normal closure of the set $\alpha(c) \beta(c)^{-1}$). In terms of group presentations, if one has $A = \langle a_i : r_j \rangle$ and $B = \langle b_k : s_{\ell} \rangle$ then $A*_C B = \langle a_i, b_k : r_j, s_{\ell}, \alpha(c) = \beta(c) \rangle$ where $c$ ranges over all $C$ (or, equivalently, over a system of generators of $C$). There are natural homomorphisms from $A$, $B$, $C$ to $A *_C B$: elements of $A$ and $B$ go directly to elements of the above presentation; elements of $C$ are first sent to $A$ or $B$ via $\alpha$ or $\beta$ and then interpreted in $A *_C B$ (the relation $\alpha(c) = \beta(c)$ ensures that either way gives the same image). We shall denote the images of these maps by $\overline{A}$, $\overline{B}$ and $\overline{C}$.

We will find ourselves in a situation where the natural map $A \longrightarrow A *_C B$ is surjective; that is, $\overline{A} = A *_C B$. Then we claim that $\beta$ is surjective. This follows from the normal form theorem for amalgamated free products (\cite[Corollary 4.4.1, p. 205]{magnuskarrasssolitar}), which states that every element in $A *_C B$ can be written uniquely in the form $w u_1 v_1 \ldots u_n v_n$ where: (i) $w \in \overline{C}$, (ii) each $u_i$ represents a nontrivial coset of $\overline{A} \mod \overline{C}$, (iii) each $v_i$ represents a nontrivial coset of $\overline{B} \mod \overline{C}$. Notice that the normal form of an element $\overline{a} \in \overline{A}$ accords to the normal form $wu_1v_1 \ldots$ with $w = \overline{a}$ (if $\overline{a} \in \overline{C}$) or with $u_1 = \overline{a}$ (if not), and similarly for $\overline{B}$. Hence the uniqueness of the normal form coupled with the assumption that $\overline{A} = A *_C B$ implies that there is only the trivial coset in $\overline{B} \mod \overline{C}$; i.e., $\overline{B} = \overline{C}$. Therefore $\beta$ is surjective.

\section{Knotted toroidal sets} \label{sec:toroidal}

In this section we prove two theorems which solve the realization problem for knotted toroidal sets. The first shows that for this class of sets, realizability as an attractor for a homeomorphism and for a flow are equivalent.

\begin{theorem} \label{lem:unknotted} Let $K \subseteq \mathbb{R}^3$ be a knotted toroidal set which is an attractor for a homeomorphism $f$ of $\mathbb{R}^3$. Then $K$ can in fact be realized as an attractor for a flow.
\end{theorem}

The second theorem, of which (2)$\Rightarrow$(1) is just a restatment of the previous one, provides two additional purely topological answers to the realizability problem for knotted toroidal sets:

\begin{theorem} \label{thm:charact}
Let $K\subseteq\mathbb{R}^3$ be a knotted toroidal set. The following conditions are equivalent:
\begin{enumerate}
    \item[(1)] $K$ is an attractor of a flow.
    \item[(2)] $K$ is an attractor of a homeomorphism.
    \item[(3)] $\check{H}^1(K;\mathbb{Z})\neq 0$ and $g(K)<\infty$.
    \item[(4)] $\pi_1(\mathbb{S}^3 \setminus K)$ is finitely generated.
\end{enumerate}
\end{theorem}

In both theorems the knottedness assumption is crucial, and the classical Whitehead continuum is a simple counterexample in this regard. It is constructed as follows. One starts with an unknotted solid torus $T_0$ and considers a homeomorphism $f$ of $\mathbb{R}^3$ which sends $T_0$ onto a solid torus along the black curve shown in the left panel of Figure \ref{fig:white}. The torus $f(T_0)$ is taken to be so thin that it is contained in the interior of $T_0$. Iterating this construction yields a compactum $K := \bigcap_{n \geq 0} f^n(T_0)$ which is the classical Whitehead continuum. It is shown in the right panel of Figure \ref{fig:white}. By construction it is an attractor for $f$, and it is toroidal (it evidently has a neighbourhood basis of solid tori; it can also be shown that it is not cellular). Since $T_0$ is unknotted and $f$ is an ambient homeomorphism all the iterates $f^n(T_0)$ are also unknotted, and so the Whitehead continuum is an unknotted toroidal set.

It is straightforward to check that $\check{H}^1(K;\mathbb{Z}) = 0$, so in particular the Whitehead continuum cannot be realized as an attractor for a flow (since these have cohomology $\mathbb{Z}$ in dimension $1$). One can also show that its complement in $\mathbb{S}^3$ is simply connected. Thus the Whitehead continuum provides a counterexample to Theorem \ref{lem:unknotted} and 
(4) $\Rightarrow$ (1) of Theorem \ref{thm:charact} in the absence of knottedness.

\begin{figure}[h]
\centering
{\includegraphics{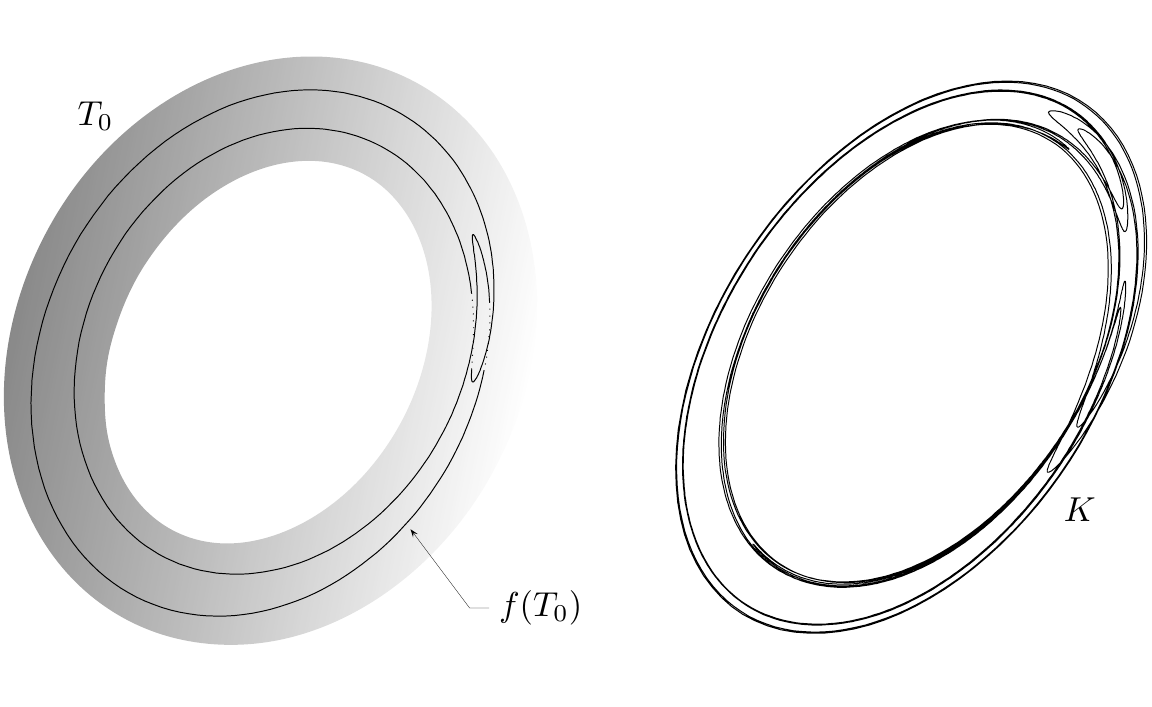}}
\caption{Whitehead's continuum as an attractor}
\label{fig:white}
\end{figure}

\begin{proof}[Proof of Theorem \ref{lem:unknotted}] As mentioned in Subsection \ref{subsec:dynamics} we can assume that $f$ is piecewise linear on $\mathbb{R}^3  \setminus  K$. Since $K$ has a neighbourhood basis of polyhedral solid tori, there is one such torus $T_0$ contained in its basin of attraction. Notice that $T_0$ must be knotted, since otherwise by iterating it with the dynamics we would obtain a neighbourhood basis of $K$ comprised on unknotted tori, contradicting the assumption that $K$ is knotted. Finally, after replacing $f$ with a power of itself we may assume that $f(T_0) \subseteq {\rm int}\ T_0$.

Denote by $M := \mathbb{S}^3  \setminus  {\rm int}\ T_0$, which is a compact $3$--manifold. This manifold contains in its interior the infinite family of disjoint, polyhedral, $2$--sided tori $S_k := f^{-k}(\partial T_0)$ for $k \geq 1$. Notice that each $S_k$ is the boundary of the solid torus $T_k := f^{-k}(T_0)$ and these form an increasing nested sequence. Thus each two consecutive $S_k$ and $S_{k+1}$ cobound the connected region $T_{k+1}  \setminus  {\rm int}\ T_k$.
\smallskip

{\it Claim.} All the $S_k$ are incompressible in $M$.

{\it Proof of claim.} Consider the pair of nested solid tori $T_0 \subseteq f^{-k}(T_0)$, so that $S_k$ is the boundary of the bigger torus. Notice that $S_k$ is contained in the interior of $M$ and separates it in two components; namely $f^{-k}(T_0) \backslash {\rm int}\ T_0$ and $\mathbb{S}^3 \backslash {\rm int}\ f^{-k}(T_0)$. If a compressing disk for $S_k$ exists, it must lie in one of those components. However $S_k$ is incompressible in both components since otherwise either: (i) there would exist a ball $B$ such that $T_0 \subseteq B \subseteq f^{-k}(T_0)$ or (ii) $f^{-k}(T_0)$ would be unknotted. The latter is not possible because $f^{-k}$ is an ambient homeomorphism and $T_0$ was taken to be unknotted. Case (i) is not possible either: iterating $B$ with the dynamics would produce a neighbourhood basis of $K$ comprised of balls and so $K$ would be cellular. This finishes the proof of the claim. $_{\blacksquare}$
\smallskip

Now we apply Theorem \ref{thm:kneserhaken} to conclude that at least two of the $2$--tori in the family $\{S_k\}$, say $S_{k_1}$ and $S_{k_2}$ with $k_2 > k_1$, are parallel. This means that the component $U$ of $M \backslash \bigcup_k S_k$ that they cobound satisfies that its closure $\overline{U}$ (which is a compact $3$--manifold) is homeomorphic to $\mathbb{T}^2 \times [0,1]$ via a homeomorphism that sends $\mathbb{T}^2 \times 0$ onto $S_{k_1}$ and $\mathbb{T}^2 \times 1$ onto $S_{k_2}$. However $\overline{U} = T_{k_2}  \setminus  {\rm int}\ T_{k_1}$, and so we see that these two solid tori are concentric. The discussion in Subsection \ref{subsec:dynamics} then shows that $K$ can be realized as an attractor for a flow.
\end{proof}

\begin{proof}[Proof of Theorem \ref{thm:charact}] The implication (1)$\Rightarrow$(2) is trivial and (2) $\Rightarrow$ (1) is the previous theorem. (3) $\Rightarrow$ (2) is \cite[Corollary 7.1]{hecyo2} and $(1)\Rightarrow (3)$ follows from the discussion in Subsection~\ref{subsec:toratt}.

(1) $\Rightarrow$ (4). Let $P \subseteq \mathcal{A}(K)$ be a compact $3$--manifold that is a positively invariant neighbourhood of $K$. As discussed in Subsection \ref{subsec:dynamics} we have $P \setminus K \cong (\partial P) \times [0,+\infty)$. In particular $P \setminus K$ can be deformed (by a strong deformation retraction) onto $\partial P$. This can be extended (by the identity) to a strong deformation retraction of $\mathbb{S}^3 \setminus K$ onto $\mathbb{S}^3 \setminus {\rm int}(P)$, and so $\pi_1(\mathbb{S}^3 \setminus K) = \pi_1(\mathbb{S}^3 \setminus {\rm int}\ P)$ (we omit basepoints since the spaces involved are connected because $K$ does not separate $\mathbb{S}^3$). Finally, since $\mathbb{S}^3 \setminus {\rm int}\ P$ is a compact $3$--manifold because $P$ is polyhedral, it has a finitely generated fundamental group.

(4) $\Rightarrow$ (1). Let $T_n$ be a neighbourhood basis of $K$ comprised of nested solid (polyhedral) tori. We can assume, perhaps after discarding finitely many of the $T_n$, that (i) each $T_n$ is knotted and (ii) each $T_{n+1}$ is essential in $T_n$; i.e. there does not exist a ball in $T_n$ which contains $T_{n+1}$. The reason for (i) is that otherwise $K$ would have a neighbourhood basis of unknotted tori and so would be unknotted, and for (ii) is that otherwise $K$ would have a neighbourhood basis of balls and so would be cellular.

Denote by $C_n := \mathbb{S}^3 \setminus {\rm int}\ T_n$. These are compact, connected $3$--manifolds with boundary the $2$--torus $\partial T_n$. Notice that the inclusion induced map $\pi_1(\partial T_n) \longrightarrow \pi_1(C_n)$ is injective because $T_n$ is knotted. The inclusion $C_n \subseteq C_{n+1}$ also induces an injective homomorphism $\pi_1(C_n) \longrightarrow \pi_1(C_{n+1})$ because of (ii) (see for example \cite[Theorem 9, p. 113]{rolfsen1}).

Observe that $\mathbb{S}^3 \setminus K$ is the union of the sequence $C_1 \subseteq C_2 \subseteq \ldots$ Since loops and homotopies between them are compactly supported, $\pi_1(\mathbb{S}^3 \setminus K)$ is the direct limit of the groups $\pi_1(C_n)$ bonded by the inclusion induced homomorphism  $\pi_1(C_n) \longrightarrow \pi_1(C_{n+1})$. These are injective as mentioned above, so each inclusion induced homomorphism $\phi_n : \pi_1(C_n) \longrightarrow \pi_1(\mathbb{S}^3 \setminus K)$ is injective too, and setting $G_n := {\rm im}\ \phi_n$ we get an ascending sequence of subgroups of $\pi_1(\mathbb{S}^3 \setminus K)$ whose union is the whole group. By assumption $\pi_1(\mathbb{S}^3 \setminus K)$ is finitely generated, so $G_n = \pi_1(\mathbb{S}^3 \setminus K)$ for large $n$ and we have $G_n = G_{n+1} = \ldots$ Since the $\phi_n$ are injective, we conclude that the inclusion induced homomorphism $\pi_1(C_n) \longrightarrow \pi_1(C_{n+1})$ is surjective (in fact an isomorphism) for large $n$.

Let $D_n := T_{n+1} \setminus {\rm int}\ T_n$. We have $C_{n+1} = C_n \cup D_n$, where $C_n \cap D_n = \partial T_n$. Since $\pi_1(\partial T_n)$ injects in both $\pi_1(C_n)$ and $\pi_1(D_n)$ (again by the knottedness and satellite conditions), the Seifert--van Kampen theorem shows that $\pi_1(C_{n+1})$ is the amalgamated product of $\pi_1(C_n)$ and $\pi_1(D_n)$ with amalgamation given by the images of $\pi_1(\partial T_n)$ in these two groups via the inclusion induced homomorphism. Moreover the natural maps from each of the groups to the amalgamated product are just the inclusion induced maps. By the previous paragraph the natural map from $\pi_1(C_n)$ to $\pi_1(C_{n+1})$ is surjective, and then by the discussion in Subsection \ref{subsec:prelim_algebra} we see that the map $\pi_1(\partial T_n) \longrightarrow \pi_1(D_n)$ is surjective, so in fact it is an isomorphism. Then by a theorem of Brown (\cite[Theorem 3.4, p. 490]{brown3}) the inclusion $\partial T_n \subseteq D_n$ extends to a homeomorphism $(\partial T_n) \times [0,1] \cong D_n$. Hence the tori $T_{n+1} \subseteq T_n$ are concentric and $K$ is an attractor for a flow by the results in Subsection \ref{subsec:dynamics}.
\end{proof}

In \cite{hecyo1} we considered a toroidal set constructed by taking an infinite connected sum of (nontrivial) knots of progressively smaller size which accumulate onto a single point $p$. This produces a knot which is wild exactly at $p$. It can be shown that its genus is infinite and so it cannot be realized as an attractor for a homeomorphism. The following corollary, previously known only for flows, generalizes this example:

\begin{corollary} A toroidal knot whose set of wild points contains an isolated point cannot be realized as an attractor for a homeomorphism. In particular, it must have an infinite genus.
\end{corollary}
\begin{proof} The result is true when $K$ is an attractor for a flow (\cite[Example 42, p. 6178]{mio5}). The corollary then follows from the theorem above.
\end{proof}

\begin{corollary} 
Let $K$ be a toroidal set with $\check{H}^1(K;\mathbb{Z})\neq \mathbb{Z}$ that is an attractor for a homeomorphism of $\mathbb{R}^3$. Then $K$ is unknotted.
\end{corollary}
\begin{proof} If $K$ were knotted, by Theorem \ref{thm:charact} it would be an attractor for a flow and so it would have $\check{H}^1(K;\mathbb{Z}) = \mathbb{Z}$.
\end{proof}

A weaker result was proved in \cite{hecyo1}, in that it only contemplated toroidal attractors with a nonfinitely generated cohomology (the third kind described in p. \pageref{eq:coh}). The methods used there were based on the genus of a toroidal set and hinged on the relation between the genus of a satellite knot and its companion, which in fact becomes noninformative when $K$ has trivial cohomology. The geometric methods used in this paper are more powerful. As a specific example, consider modifying the construction of the Whitehead continuum as follows. Start with a knotted torus $T_0$. Then place a torus $T_1 \subseteq {\rm int}\ T_0$ in the self-linking pattern shown in the left panel of Figure \ref{fig:white} and keep repeating the construction to obtain a sequence $T_n$. (Unlike the original construction of the continuum, we do not claim that this can be done by iterating a homeomorphism of $\mathbb{R}^3$). The intersection $K := \bigcap_{n \geq 0} T_n$ is a knotted toroidal set which still has $\check{H}^1(K;\mathbb{Z}) = 0$, so by the corollary above it cannot be realized as an attractor for a homeomorphism of $\mathbb{R}^3$. (So in fact the construction cannot be done by iterating a homeomorphism of $\mathbb{R}^3$).

\section{Attractor-repeller decompositions and toroidal sets}\label{sec:att-rep}

Let $f:\mathbb{S}^3\longrightarrow\mathbb{S}^3$ be a homeomorphism and suppose that $K$ is an attractor for  $f$. Then $K'=\mathbb{S}^3\setminus\mathcal{A}(K)$ is an attractor for the backwards dynamics $f^{-1}$; i.e. a repeller for $f$ known as the dual repeller of $K$. The region of repulsion of $K'$ (for $f$) is defined to be its region of attraction for $f^{-1}$. Finally, the pair $(K,K')$ is called an attractor-repeller decomposition of $\mathbb{S}^3$. If $K$ is toroidal we can find a solid torus $T\subseteq\mathcal{A}(K)$ that is a neighbourhood of $K$, and then $\partial T$ is a  $2$-torus that separates $K$ and $K'$. In this section we study the converse situation, i.e., what can we say about an attractor $K$ and its dual repeller $K'$ if we know that they are separated by a $2$-torus? The following result gives a fairly complete description of what happens in this situation.

\begin{theorem}\label{thm:decomp}
Let $K\subseteq\mathbb{S}^3$ be an attractor for a homeomorphism and $K'$ its dual repeller. Suppose that there exists a $2$-torus in $\mathbb{S}^3$ that separates $K$ and $K'$. Then $\check{H}^1(K;\mathbb{Z})\cong\check{H}^1(K';\mathbb{Z})$ and precisely one of the following possibilities holds:
\begin{enumerate}
    \item $K$ and $K'$ are cellular.
    \item One of $K$ and $K'$ is a knotted toroidal set, while the other set is not toroidal.
    \item Both $K$ and $K'$ are unknotted toroidal sets.
\end{enumerate}
Moreover, in cases (1) and (2) the pair $(K,K')$ is an attractor-repeller decomposition of a flow on $\mathbb{S}^3$.
\end{theorem}

\begin{proof}
Let $S$ be a $2$--torus that separates $K$ and $K'$. Notice that by \cite[Claim~2, p. 10393]{hecyo1} we may assume that $S$ is a polyhedral torus. Then, a classical theorem of Alexander (\cite[1., p. 107]{rolfsen1}) ensures that one of the components $U$ of $\mathbb{S}^3 \setminus  S$ satisfies that $T=\overline{U}$ is a polyhedral solid torus whose boundary is $S$. This solid torus $T$ must contain in its interior either $K$ or $K'$. We assume that $T$ contains $K$, the other case being completely analogous. Since $T$ does not intersect $K'$ it follows that $T\subseteq\mathcal{A}(K)$ and, hence, maybe after replacing $f$ by a suitable power, we have that $f(T)\subseteq{\rm int}\; T$. Then, if we define $T_n:=f^{n}(T)$ we have that $\{T_n\}_{n\geq 0}$ is a basis of neighbourhoods of $K$ comprised of solid tori. This ensures that $K$ is either a cellular set or a toroidal set. 

Suppose first that $K$ is cellular. Then, there exists a polyhedral closed ball $B\subseteq T$ that is a neighbourhood of $K$. Then, by the Sch\"{o}enflies theorem for bicollared spheres we have that $B'=\mathbb{S}^3\setminus{\rm int}\;B$ is a polyhedral closed ball and $B'$ is a neighbourhood of $K'$ contained in its region of repulsion. By iterating $B'$ backwards we obtain a basis of neighbourhoods of $K'$ comprised of closed balls. It follows that $K'$ is also cellular.

 Suppose now that $K$ is an unknotted toroidal set. We may assume without loss of generality that $T$ is unknotted. Then $T'=\mathbb{S}^3\setminus{\rm int}\; T$ is an unknotted solid torus  that is a neighbourhood of $K'$ contained in its basin of repulsion and $f^{-1}(T')\subseteq {\rm int}\; T'$. Then, by the discussion in Subsection \ref{subsec:plapprox} $K'$ is a repeller of a homeomorphism $h$ that is piecewise linear in $\mathbb{S}^3\setminus K'$ and such that $T'$ is in the basin of repulsion of $K'$ and $h^{-1}(T')\subseteq{\rm int}\; T'$. By iterating backwards $T'$ with $h$, we obtain a basis of neighbourhoods of $K'$ comprised of unknotted polyhedral solid tori. Observe that $K'$ cannot be cellular since otherwise the previous arguments would ensure the cellularity of $K$ in contradiction with the assumption. Therefore, $K'$ is also an unknotted toroidal set.  
 
Finally we assume that $K$ is a knotted toroidal set and show that $K'$ is not a toroidal set. We argue by contradiction. Suppose that $K'$ is toroidal and let $T''\subseteq\mathbb{S}^3\setminus T\subseteq\mathcal{R}(K')$ be a polyhedral solid torus that is a neighbourhood of $K'$. Choosing $h$ as before we have that there exists $k>0$ such that $T''\subseteq {\rm int}\; h^k(T'')$ and $h^k(\partial T'')\subseteq {\rm int}\; T$. The second inclusion comes from the fact that the construction of $h$ ensures that $T$ is a neighbourhood of the dual attractor of $K'$ for this new $h$ contained in its basin of attraction. Notice that $h^k(\partial T'')$ is compressible in ${\rm int}\; T$ since otherwise we would have an injection of $\pi_1(h^k(\partial T''))=\mathbb{Z}\oplus\mathbb{Z}$ into $\pi_1(T)=\mathbb{Z}$. Since $T\subseteq\mathbb{S}^3\setminus{\rm int}\; T''$ it follows that $h^k(\partial T'')$ is also compressible in $\mathbb{S}^3\setminus{\rm int}\; T''$ and we have that either $T''$ is contained in a ball inside $h^k(T'')$ or $h^k(T'')$ is an unknotted solid torus. In the first case we would have that $K'$ and $K$ are both cellular and in the second case we would have that $K$ and $K'$ are both unknotted toroidal sets in contradiction with the assumption. Therefore $K'$ cannot be a toroidal set.
 
Observe that if $K$ is cellular or a knotted toroidal set, \cite[Theorem~1.1]{garay1} and Theorem~\ref{lem:unknotted} ensure that $K$ is an attractor for a flow and, by \cite[Theorem~45]{mio5}  $K$ is an attractor for a flow with basin of attraction $\mathcal{A}(K)=\mathbb{S}^3\setminus K'$. Therefore, $(K,K')$ is an attractor-repeller  decomposition for this flow. 

We conclude by showing that $\check{H}^1(K;\mathbb{Z})\cong\check{H}^1(K',\mathbb{Z})$. In case that $K$ and $K'$ are cellular the isomorphism follows because both $K$ and $K'$ have trivial first \v Cech cohomology group. Suppose then that at least one of them, say $K$, is a toroidal set (if $K'$ is toroidal the argument is completely analogous). As done several times before, construct a neighbourhood basis $\{T_n := f^n(T)\}$ of $K$ by iterating some solid torus $T \subseteq \mathcal{A}(K)$ which is a neighbourhood of $K$ (and possibly replacing $f$ with a power so that the $T_n$ are nested). Observe that then $\{\mathbb{S}^3 \setminus {\rm int}\; f^{-n}(T)\}$ is a nested neighbourhood basis of $K'$, the dual repeller of $K$. Although these sets need not be tori (this will happen if and only if $T$ is unknotted), by Alexander duality they have the cohomology of a solid torus. Hence, for each $n$ the inclusion $\mathbb{S}^3 \setminus {\rm int}\; f^{-n}(T)\subseteq\mathbb{S}^3 \setminus {\rm int}\; f^{-n-1}(T)$ induces in $H^1$ multiplication by a non-negative integer $w_n$ (by choosing the right orientations) that we shall also refer to as the winding number. We shall make use of the following fact:

%Let $T \subseteq \mathcal{A}(K)$ be a solid torus which is a neighbourhood of $K$ and Then, $K$ is an attractor for a flow and, hence, $\check{H}^1(K;\mathbb{Z})=\mathbb{Z}$. Then we have that
%\[
%\check{H}^1(K;\mathbb{Z})\cong H_1(\mathcal{R}(K');\mathbb{Z})\cong H^1(\mathcal{R}(K');\mathbb{Z})\cong \check{H}^1(K';\mathbb{Z})
%\]
%where the first isomorphism comes from Alexander duality, the second from the Universal Coefficient Theorem and the third by \cite[Theorem~2]{pacoyo1}. The case when $K'$ is a knotted toroidal set is completely analogous, so it only remains to consider the case when $K$ and $K'$ are both unknotted toroidal sets. Suppose again that $T$ is unknotted and let $\{T_n=f^n(T)\}$ that is a basis of $K$ comprised of unknotted solid tori. Let $w$ be the winding number of $T_1$ inside $T_0$ and notice that the construction of the basis ensures that the winding number of $T_{i+1}$ inside $T_i$ is also $w$ for every $i\geq 0$.

{\it Claim.} Let $T'\subseteq T\subseteq\mathbb{S}^3$ be a pair of nested solid tori. Then winding number of $T'$ inside $T$ coincides with the winding number of $\mathbb{S}^3\setminus{\rm int}\; T$ inside $\mathbb{S}^3\setminus {\rm int}\; T'$.

{\it Proof of Claim.} Let $w$ be the winding number of $T'$ inside $T$. It follows from the long exact sequence of cohomology of the pair $(T,T')$ that $H^2(T,T';\mathbb{Z})= \mathbb{Z}_w$ (integers modulo $w$). By Alexander duality it follows that 
\[
\mathbb{Z}_w=H^2(T,T';\mathbb{Z})\cong H_1(\mathbb{S}^3\setminus T',\mathbb{S}^3\setminus T;\mathbb{Z}).
\]
Taking this into account in the long exact sequence of homology of the pair $(\mathbb{S}^3\setminus T',\mathbb{S}^3\setminus T)$ it follows that the inclusion $\mathbb{S}^3\setminus T\subseteq \mathbb{S}^3\setminus T'$ induces the multiplication by $w$ in $H_1$ and, by the Universal Coefficient Theorem, also in $H^1$. The claim follows from the fact that if $M$ is a compact manifold with boundary, the inclusion ${\rm int}\; M\subseteq M$ induces isomorphisms in cohomology. $_\blacksquare$

As a consequence of the claim, to prove that $\check{H}^1(K)$ and $\check{H}^1(K')$ are isomorphic, it is sufficient to see that for each $n\geq 0$, the winding number of $\mathbb{S}^3\setminus{\rm int}\; f^{-n-1}(T)$ inside $\mathbb{S}^3\setminus{\rm int}\; f^{-n}(T)$ coincides with the winding number of $\mathbb{S}^3\setminus{\rm int}\; f^{n}(T)$ inside $\mathbb{S}^3\setminus{\rm int}\; f^{n+1}(T)$. But this follows in a straightforward way from the fact that the homeomorphism $f^{2n+1}$ maps the former pair onto the latter.

\end{proof}

The theorem above corrects \cite[Theorem 3.9, p. 10393]{hecyo1}, which stated that when $\check{H}^1(K) \neq \mathbb{Z}$ case (3) must occur. In the course of the proof (namely, in Claim 3) implicit use was made of the assumption that $\check{H}^1(K) \neq 0$ as well, which was not stated explicitly in the theorem. Of course, this additional assumption rules out case (1) in Theorem \ref{thm:decomp} above, leaving only (3).

In case that the $2$-torus $S$ is polyhedral, the three cases described in Theorem~\ref{thm:decomp} can be studied in terms of the compressibility properties of $S$ in $\mathcal{A}(K)\setminus K$ and $\mathbb{R}^3\setminus K$. In particular, we have that the case (1) is equivalent to $S$ being compressible in $\mathcal{A}(K)\setminus K$. Situation (2) occurs whenever $S$ is incompressible in $\mathbb{S}^3\setminus K$ and (3) happens whenever $S$ is compressible in  $\mathbb{S}^3\setminus K$ while incompressible in $\mathcal{A}(K)\setminus K$. From these considerations the next result follows.

\begin{corollary}
Let $K\subseteq\mathbb{S}^3$ be an attractor for a homeomorphism and $K'$ its dual repeller. Suppose that $K$ is a toroidal set. Then there exists a polyhedral $2$-torus in $\mathbb{S}^3$ that separates $K$ and $K'$ and  is incompressible in $\mathcal{A}(K)\setminus K$. Conversely, if  there exists a polyhedral $2$-torus in $\mathbb{S}^3$ that separates $K$ and $K'$ and is incompressible in $\mathcal{A}(K)\setminus K$ then at least one between $K$ and $K'$ is toroidal.
\end{corollary}

\section{Beyond toroidal sets} \label{sec:generalize}

The techniques used above can in fact be generalized for arbitrary attractors $K$ in $\mathbb{R}^3$ (or more general $3$--manifolds). As far as the realization problem is concerned we may concentrate on connected sets which do not separate $\mathbb{R}^3$. This can be justified as follows. First, we already mentioned that any attractor has finitely many connected components each of which is an attractor (possibly for a suitable power of the dynamics), so a necessary condition for a compactum to be realizable as an attractor is that each of its components can be so realized. Also, if an attractor $K$ separates $\mathbb{R}^3$ then it also separates $\mathbb{S}^3$ and does so in finitely many connected components $U_1, \ldots, U_r$ because $\check{H}^2(K;\mathbb{Z}_2)$ is finitely generated. Then, and perhaps after replacing the dynamics with a suitable power that leaves each $U_i$ invariant, $K \cup \bigcup_{j \neq i} U_j$ (which is the result of capping all but one of the cavities in $K$) is an attractor. Thus for a separating compactum $K$ to be realizable as an attractor it is necessary that each of the nonseparating compacta obtained by succesively capping its holes be realizable as an attractor as well.

In the case of knotted toroidal attractors all our arguments hinged on choosing a torus $T$ in $\mathcal{A}(K)$ whose boundary was incompressible in $\mathbb{S}^3 \backslash K$ or, equivalently, incompressible both in $\mathcal{A}(K) \backslash K$ and $\mathbb{S}^3 \backslash {\rm int}\ T$. The first was ensured by the non-cellularity condition on $K$ whereas the second reflected the fact that $K$ is knotted. In the general case we can only achieve incompressibility in $\mathcal{A}(K) \backslash K$. This is the content of the following lemma:

\begin{lemma} \label{lem:incompressible}
Let $K\subseteq\mathbb{R}^3$ be an attractor for a homeomorphism. Assume that $K$ is connected and does not separate $\mathbb{R}^3$. Then $K$ has a compact neighbourhood $N\subseteq\mathcal{A}(K)$ which is a  $3$-manifold such that $\partial N$ is connected and incompressible in $\mathcal{A}(K) \backslash K$.
\end{lemma}

We make use of the following fact. Suppose $U \subseteq \mathbb{R}^3 \subseteq \mathbb{S}^3$ is open and connected and $S \subseteq U$ is a closed connected surface. If $H_2(U;\mathbb{Z}_2) = 0$, then the bounded component of $\mathbb{R}^3 \backslash S$ is contained in $U$. This is intuitively clear but can be proved formally as follows. Since $H_2(U;\mathbb{Z}_2) = 0$, Alexander duality shows that $\mathbb{S}^3 \backslash U$ is connected. Also $S$, being a closed connected surface, separates $\mathbb{S}^3$ into two connected components $V_1$ and $V_2$; say that $V_1$ is the component that contains $\infty$ (and so $V_2$ is in fact the bounded component of $\mathbb{R}^3 \backslash S$). The connected set $\mathbb{S}^3 \backslash U$ is contained in $\mathbb{S}^3 \backslash S = V_1 \uplus V_2$, so it must be contained in one of the components $V_i$; since $\mathbb{S}^3 \backslash U$ contains $\infty$, it must be $\mathbb{S}^3 \backslash U \subseteq V_1$. Thus $\mathbb{S}^3 \backslash V_1 \subseteq U$. But $\mathbb{S}^3 \backslash V_1 = V_2 \cup S$, which proves $V_2 \subseteq U$.

\begin{proof} Start with any polyhedral compact $3$--manifold $N$ which is a neighbourhood of $K$ contained in $\mathcal{A}(K)$.

Now compress $\partial N$ in $\mathcal{A}(K) \backslash K$ as much as possible. Compressing along a disk $D$ contained in $N \backslash K$ may disconnect $N$. Compressing along a disk $D$ contained in $\mathcal{A}(K) \backslash {\rm int}\ N$ is equivalent to pasting a $2$--handle onto $N$ and may disconnect its boundary. In either case the process is finite and produces a new $N$ whose boundary is perhaps not connected but each of whose components is incompressible or a $2$--sphere. If $N$ is not connected, we discard all of its components except for the one that contains $K$.

It remains to prove that we can take $\partial N$ to be connected. The basic idea is to exploit that $\mathcal{A}(K)$ has no cavities and then take the ``outermost'' component of $\partial N$. We formalize this as follows.

Denote by $D_1, \ldots, D_r$ the components of the boundary of $\partial N$ and by $E_i$, $F_i$ the closures of the bounded and unbounded components of $\mathbb{R}^3 \backslash D_i$ respectively. Observe that for each $i$ the attractor $K$ is contained in either $E_i$ or $F_i$.

(i) Claim: at least one $E_i$ contains $K$. 

{\it Proof of claim.} since $N$ is connected, for each $i$ we have that $N$ is contained in either $E_i$ or $F_i$, and in fact $N$ is the intersection of the corresponding $E_i$ or $F_i$. Now if the claim above were false, then we would have $N = \bigcap_i F_i$. However, for large enough $r_0$ we have that each $F_i$ contains the sphere of radius $r$ for all $r \geq r_0$, and so ${\rm int}\ N$ would be unbounded, which contradicts the fact that $N$ is compact.

(ii) Claim: if $E_i$ and $E_j$ contain $K$, then one is contained in the other.

{\it Proof of claim.} This is clear: both are connected sets and their frontiers $D_i$ and $D_j$ are connected and disjoint, so one must be contained in the other.

In fact it can be shown that only one $E_i$ can contain $K$, but we shall not need that. It suffices to look at all the $E_i$ that contain $K$ (there is at least one by the first claim) and take the biggest one (this is meaningful by the second claim). This $E_{i_0}$ is a compact $3$--manifold neighbourhood of $K$ with boundary $D_{i_0}$, which is incompressible in $\mathcal{A}(K) \backslash K$. It remains to show that it is contained in $\mathcal{A}(K)$. The inclusion $K \subseteq \mathcal{A}(K)$ induces isomorphisms in \v{C}ech cohomology in all dimensions (\cite{pacoyo1}). Since $K$ is nonseparating, $\check{H}^2(K;\mathbb{Z}_2) = 0$ and so $H^2(\mathcal{A}(K);\mathbb{Z}_2) = 0$ as well. By the universal coefficient theorem, $H_2(\mathcal{A}(K);\mathbb{Z}_2) = 0$. Then the fact stated before the proof leads to $E_{i_0} \subseteq \mathcal{A}(K)$.
\end{proof}

\begin{remark}
    Lemma~\ref{lem:incompressible} can be proved in an easier way by invoking \cite[Proposition~1.2]{Swarup}. This result ensures that if a non-compact orientable $3$-manifold $M$ without boundary has at least two ends, then it contains a connected incompressible surface that does not bound a region inside $M$. Notice that $\mathcal{A}(K)\setminus K$ is orientable, does not have boundary and has exactly two ends. The latter follows taking into account that, as we discussed previously, $H_2(\mathcal{A}(K),\mathbb{Z}_2)=0$ which ensures that $\mathbb{S}^3\setminus \mathcal{A}(K)$ is connected. Hence, $\mathbb{S}^3\setminus(\mathcal{A}(K)\setminus K)=(\mathbb{S}^3\setminus\mathcal{A}(K))\cup K$ has two connected components that are in correspondence with the ends of $\mathcal{A}(K)\setminus K$. Then, by \cite[Proposition~1.2]{Swarup} there exists a connected incompressible surface $S$ that does not bound a region inside $\mathcal{A}(K)\setminus K$. The desired neighbourhood $N$ is just the closure of the bounded component of $\mathbb{R}^3\setminus S$ that is contained in $\mathcal{A}(K)$ by the observation before the proof of Lemma~\ref{lem:incompressible}.

    In spite of the fact that this proof is easier that the given proof of Lemma~\ref{lem:incompressible} as is presented here, we believe that our proof suits better our purposes because we give an explicit construction of $N$ from any $3$-manifold neighbourhood of $K$ contained in $\mathcal{A}(K)$. In addition we will refer to this construction later. 
\end{remark}

A consequence of Lemma~\ref{lem:incompressible} is that cellularity of an attractor is characterized by the simple connectivity of its complement in its region of attraction.

\begin{corollary}\label{coro:cell}
Let $K\subseteq\mathbb{R}^3$ be an attractor for a homeomorphism. Assume that $K$ is connected and does not separate $\mathbb{R}^3$. The following are equivalent: 
\begin{enumerate}
    \item $K$ is cellular.
    \item $\mathcal{A}(K)\setminus K$ is simply connected.
\end{enumerate}
\end{corollary}

\begin{proof}
(1) $\Rightarrow$ (2). Suppose that $K$ is cellular and let $K'$ be its dual repeller in $\mathbb{S}^3$ after extending $f$ leaving $\infty$ fixed. Reasoning as in the proof of Theorem~\ref{thm:decomp} it follows that $K'$ is also cellular. Let $\{B_n\}$ and $\{B'_n\}$ be nested bases of neighbourhoods comprised by polyhedral cells of $K$ and $K'$ respectively. For each $n$, the Seifert--van Kampen theorem ensures that $\mathbb{S}^3\setminus (B_n\cup B'_n)$ is simply connected. The simple connectivity of $\mathcal{A}(K)\setminus K$ follows from the fact that any loop in $\mathcal{A}(K)\setminus K$ is contained in  $\mathbb{S}^3\setminus (B_n\cup B'_n)$ for $n$ sufficiently large since
\[
\mathcal{A}(K)\setminus K=\mathbb{S}^3\setminus (K\cup K')=\bigcup_{n}\left(\mathbb{S}^3\setminus (B_n\cup B'_n)\right)
\]
and $\mathbb{S}^3\setminus (B_n\cup B'_n)\subseteq\mathbb{S}^3\setminus (B_{n+1}\cup B'_{n+1})$ for each $n$.

(2) $\Rightarrow$ (1). Suppose now that $\mathcal{A}(K)\setminus K$ is simply connected. By Lemma~\ref{lem:incompressible} $K$ has a neighbourhood $N\subseteq\mathcal{A}(K)$ that is a $3$-manifold such that $\partial N$ is connected and incompressible in $\mathcal{A}(K) \backslash K$. The simple connectivity of $\mathcal{A}(K)\setminus K$ ensures that $\partial N$ must be homeomorphic to a sphere since otherwise its fundamental group (which would be nontrivial) would inject into $\pi_1(\mathcal{A}(K)\setminus K)=0$ leading to a contradiction. Then the Sch\"{o}enflies theorem for bicollared spheres ensures that $N$ is homeomorphic to a ball and, since $N\subseteq\mathcal{A}(K)$, we get that $K$ is cellular.
\end{proof}

\begin{remark}
 Using strong results from the theory of upper-semicontinuous decompositions it can be proved that an attractor $K\subseteq\mathbb{R}^3$ is cellular if and only if $\mathcal{A}(K)\setminus K$ is homeomorphic to $\mathbb{S}^2\times\mathbb{R}$. Indeed, since both $K$ and $K'$ are cellular, \cite[Theorem~2, pg. 23]{dav} together with \cite[Corollary~2A, pg. 36]{dav} ensure that $\mathcal{A}(K)\setminus K=\mathbb{S}^3\setminus (K\cup K')$ is homeomorphic to $\mathbb{S}^3\setminus \{p,q\}$ where $p,q\in\mathbb{S}^3$ and, hence, to the product $\mathbb{S}^2\times\mathbb{R}$. The converse follows from Corollary~\ref{coro:cell} since $\mathbb{S}^2\times\mathbb{R}$ is simply connected.
\end{remark}

The following result can be regarded as a generalization of Theorem~\ref{lem:unknotted}.

\begin{theorem}\label{thm:nbhinc}
Let $K\subseteq\mathbb{R}^3$ be an attractor for a homeomorphism. Suppose that $K$ is connected, does not separate $\mathbb{R}^3$ and has a neighbourhood $N\subseteq\mathcal{A}(K)$ that is a compact $3$-manifold such that $\partial N$ is incompressible in $\mathbb{S}^3\setminus{\rm int}\;  N$. Then $K$ is an attractor for a flow. 
\end{theorem}

\begin{proof} As discussed in Subsection \ref{subsec:plapprox} we can replace $f$ with a homeomorphism which is piecewise linear on $\mathbb{R}^3 \setminus K$ and still has $K$ as an attractor with $N$ contained in its basin of attraction.

{\it Claim.} $K$ has a neighbourhood $P\subseteq\mathcal{A}(K)$ that is a compact $3$-manifold and such that $\partial P$ is connected and incompressible in $\mathbb{S}^3\setminus K$.

{\it Proof of claim.} If $\partial N$ is compressible in $\mathbb{S}^3\setminus K$ the assumption ensures that every compressing disk for $\partial N$ is contained in $N\setminus K$. Reasoning as in the proof of Lemma~\ref{lem:incompressible}, after compressing $\partial N$ as much as possible in $N\setminus K$, we can find a neighbourhood $P\subseteq N\subseteq\mathcal{A}(K)$ of $K$ that is a compact $3$-manifold and such that $\partial P$ is connected and incompressible in $P\setminus K$. However, since $P$ has been obtained by cutting $N$ along $\partial N$, that is incompressible in $\mathbb{S}^3\setminus{\rm int}\; N$ by assumption, it follows that $\partial P$ is incompressible in $\mathbb{S}^3\setminus{\rm int}\; P$. Thereofore, $\partial P$ is incompressible in $\mathbb{S}^3\setminus K$ as claimed. $_{\blacksquare}$

Since $\mathbb{S}^3\setminus K$ is invariant and $\partial P$ incompressible in $\mathbb{S}^3\setminus K$ it follows all its iterates are also incompressible in $\mathbb{S}^3\setminus K$. Maybe after replacing $f$ by a suitable power we have that $P\subseteq {\rm int}\; f^{-1}(P)$. Then the family $S_k:=f^{-k}(\partial P)$ is an infinite family of disjoint connected incompressible surfaces in the compact manifold $\mathbb{S}^3\setminus{\rm int}\; P$ and the result follows along the same lines of the proof of Theorem~\ref{lem:unknotted} as a consequence of Theorem~\ref{thm:kneserhaken}.
\end{proof}

To motivate the next result consider again for a moment the case when $K$ is a knotted toroidal attractor for a homeomorphism $f$. Let $T_0$ be a solid torus neighbourhood of $K$ in $\mathcal{A}(K)$. As in the proof of Theorem \ref{lem:unknotted} its boundary $\partial T_0$ is incompressible both in $T_0$ and in $\mathbb{S}^3 \setminus {\rm int}\ T_0$. Then writing $\mathbb{S}^3 \setminus K = (\mathbb{S}^3 \setminus {\rm int}\ T_0) \cup (T_0  \setminus K)$ shows, by the Seifert-van Kampen theorem, that the inclusion induced homomorphism $\pi_1(T_0 \setminus K) \longrightarrow \pi_1(\mathbb{S}^3 \setminus K)$ is injective. The same is then true of every inclusion induced map $\pi_1(f^{-r}(T_0) \setminus K) \longrightarrow \pi_1(\mathbb{S}^3 \setminus K)$ because $f^{-r}(\mathbb{S}^3 \setminus K) = \mathbb{S}^3 \setminus K$. Since any loop in $\mathcal{A}(K) \setminus K$ is contained in a set of the form $f^{-r}(T_0) \setminus K$ for large enough $r$, it follows that the inclusion induced map $\pi_1(\mathcal{A}(K) \setminus K) \longrightarrow \pi_1(\mathbb{S}^3 \setminus K)$ is injective as well. The following theorem shows that this algebraic condition, which as we have just seen holds for a knotted toroidal attractor, is in fact enough to reach the same conclusion as in the preceding theorem:

\begin{theorem}\label{thm:inject}
Let $K\subseteq\mathbb{R}^3$ be an attractor for a homeomorphism. Suppose that $K$ that does not separate $\mathbb{R}^3$ and the homomorphism that $i_*:\pi_1(\mathcal{A}(K)\setminus K)\longrightarrow\pi_1(\mathbb{R}^3\setminus K)$ induced by the inclusion is injective.
Then $K$ is an attractor for a flow.
\end{theorem}

\begin{proof}
Let $N\subseteq\mathcal{A}(K)$ be a neighbourhood of $K$ that is a compact $3$-manifold such that $\partial N$ is connected and incompressible in $\mathcal{A}(K)\setminus K$ that exists by Lemma~\ref{lem:incompressible}. If $\partial N$ is a $2$-sphere then $N$ must be a ball and $K$ must be cellular and, hence, an attractor for a flow. Assume that $\partial N$ is not a $2$-sphere. Let $j:\partial N\hookrightarrow \mathcal{A}(K)\setminus K$ and $k:\partial N\hookrightarrow \mathcal{A}(K)\setminus K$ be the inclusion maps. The incompressibility of $\partial N$ in $\mathcal{A}(K)\setminus K$ is equivalent to the injectivity of the induced homomorphism $j_*:\pi_1(\partial N)\longrightarrow \pi_1(\mathcal{A}(K)\setminus K)$. This fact together with the assumption on the injectivity of $i_*$ give the injectivity of the homorphism $k_*$, being $k_*$ the composition of the two. This last injectivity is equivalent to the incompressibility  of $\partial N$ in $\mathbb{S}^3\setminus K$, which implies the incompressibility of $\partial N$ in $\mathbb{S}^3\setminus{\rm int}\; N$. The result follows from Theorem~\ref{thm:nbhinc}.
\end{proof}

Our last theorem is an appropriate generalization of (1) $\Leftrightarrow$ (4) of Theorem \ref{thm:charact}. Notice that we now require that $\pi_1(\mathcal{A}(K) \setminus K)$ be finitely generated, rather than $\pi_1(\mathbb{S}^3 \setminus K)$ as in Theorem \ref{thm:charact}. This was to be expected, since we already discussed that the Whitehead continuum has $\pi_1(\mathbb{S}^3 \setminus K)$ finitely generated (in fact, trivial) but is not an attractor for a flow.

\begin{theorem} Let $K \subseteq \mathbb{R}^3$ be a connected, nonseparating attractor for a homeomorphism. The following are equivalent:
\begin{itemize}
	\item[(1)] $K$ can be realized as an attractor for a flow (with the same region of attraction).
	\item[(2)] $\pi_1(\mathcal{A}(K)\setminus K)$ is finitely generated.
\end{itemize}
\end{theorem}
\begin{proof} (1) $\Rightarrow$ (2). As mentioned in Subsection \ref{subsec:dynamics} $K$ has a positively invariant neighbourhood $P \subseteq \mathcal{A}(K)$ which is a compact $3$--manifold and such that the flow is transverse to $\partial P$. Then the flow provides a homeomorphism $\mathcal{A}(K) \setminus K \cong (\partial P) \times \mathbb{R}$, and so $\pi_1(P \setminus K) = \pi_1(\mathcal{A}(K) \setminus K) = \pi_1(\partial P)$ which is finitely generated because $\partial P$ is a compact surface.

(2) $\Rightarrow$ (1). Let $N \subseteq \mathcal{A}(K)$ be a neighbourhood of $K$ as in Lemma \ref{lem:incompressible}. Namely, $N$ is a compact polyhedral $3$--manifold with $\partial N$ connected and incompressible in $\mathcal{A}(K) \setminus K$. In particular the inclusion induced map $\pi_1(\partial N) \longrightarrow \pi_1(\mathcal{A}(K) \setminus K)$ is injective. Replacing $f$ with a suitable power we may assume that $N$ is positively invariant. Since $f$ is a homeomorphism of $\mathcal{A}(K)\setminus K$, not only the inclusion of $\partial N$ in $\mathcal{A}(K) \setminus K$ but also of all its iterates induces injective homomorphisms on the fundamental groups. Of course then the same is true of the inclusion of any iterate of $\partial N$ in any subset of $\mathcal{A}(K) \setminus K$.

Let us set $D := f^{-1}(N) \setminus {\rm int}\ N$, which is a compact and connected $3$--manifold whose boundary consists of two copies of $\partial N$. For each $r \geq 0$ we consider $F_r := f^{-r}(D) \cup f^{-r+1}(D) \cup \ldots \cup f^{r-1}(D) \cup f^{r}(D)$. Observe that $F_{r+1}$ is the union of $F_r$ and two copies of $D$ (namely $f^{r+1}(D)$ and $f^{-(r+1)}(D)$) pasted onto $F_r$ along its two boundary components. Since the latter are just iterates of $\partial N$, the previous paragraph implies that their inclusions in $F_r$, $f^{r+1}(D)$ and $f^{-(r+1)}(D)$ all induce injective homomorphisms on the fundamental group. All this implies, via the Seifert-van Kampen theorem, that the inclusion $F_r \subseteq F_{r+1}$ induces an injective map $\pi_1(F_r) \longrightarrow \pi_1(F_{r+1})$.

The rest of the argument goes as in the case of toroidal sets. Since $\mathcal{A}(K)\setminus K$ is the union of the $F_r$, its fundamental group is the direct limit of $\pi_1(F_r)$ bounded by the (injective) inclusion induced maps. The assumption that $\mathcal{A}(K)\setminus K$ be finitely generated implies that these maps must be isomorphisms for big enough $r$, and in turn by the algebraic observation of Subsection \ref{subsec:prelim_algebra} the inclusion $\partial N \subseteq D$ must induce an isomorphism in $\pi_1$, and so $D$ must be a product. Thus $N$ and $f^{-1}(N)$ are concentric and we are finished.
\end{proof}

 %Suppose that $\check{H}^1(K;\mathbb{Z})$ is finitely generated. In particular, $\check{H}^1(K;\mathbb{Z})$ is either $0$ or $\mathbb{Z}$. Then, 
%\[
%\check{H}^1(K;\mathbb{Z})\cong H_1(\mathcal{R}(K');\mathbb{Z})\cong H^1(\mathcal{R}(K');\mathbb{Z})
%\]
%Where the first isomorphism comes from Alexander duality and the second from the Universal Coefficient Theorem. Moreover, by \cite[Theorem~2]{pacoyo1} $H^1(\mathcal{R}(K');\mathbb{Z})\cong\check{H}^1(K';\mathbb{Z})$.  The converse follows interchanging the roles of $K$ and $K'$. 

%Suppose now that $\check{H}^1(K;\mathbb{Z})$ is not finitely generated. 


\bibliographystyle{plain}

\bibliography{biblio}

\begin{thebibliography}{10}

\bibitem{hecyo1}
H.~Barge and J.~J. S\'{a}nchez-Gabites.
\newblock Knots and solenoids that cannot be attractors of self-homeomorphisms
  of $\mathbb{R}^3$.
\newblock {\em Int. Math. Res. Not.}, 13:10373--10407, 2021.

\bibitem{hecyo2}
H.~Barge and J.~J. S\'{a}nchez-Gabites.
\newblock The geometric index and attractors of homeomorphisms of $\mathbb
  {R}^3$.
\newblock {\em Ergodic Theory Dynam. Systems}, 43:50–77, 2022.

\bibitem{hecyo3}
H.~Barge and J.J. S\'{a}nchez-Gabites.
\newblock The realization problem of non-connected compacta as attractors.
\newblock {\em Topology Appl.}
\newblock To appear.

\bibitem{beck1}
A.~Beck.
\newblock {\em Continuous flows in the plane}.
\newblock Die Grundlehren der mathematischen Wissenschaften, Band 201.
  Springer-Verlag, New York-Heidelberg, 1974.
\newblock With the assistance of Jonathan Lewin and Mirit Lewin.

\bibitem{brown3}
E.~M. Brown.
\newblock Unknotting in {$M^2 \times I$}.
\newblock {\em Trans. Amer. Math. Soc.}, 123(2):480--505, 1966.

\bibitem{chewningowen1}
W.~C. Chewning and R.~S. Owen.
\newblock Local sections of flows on manifolds.
\newblock {\em Proc. Amer. Math. Soc.}, 49(1):71--77, 1975.

\bibitem{crovisierrams1}
S.~Crovisier and M.~Rams.
\newblock {IFS attractors and Cantor sets}.
\newblock {\em Topology Appl.}, 153:1849--1859, 2006.

\bibitem{dav}
R.~J. Daverman.
\newblock {\em Decompositions of manifolds}, volume 124 of {\em Pure and
  Applied Mathematics}.
\newblock Academic Press, Inc., Orlando, FL, 1986.

\bibitem{duvallhusch1}
P.~F. Duvall and L.~S. Husch.
\newblock Attractors of iterated function systems.
\newblock {\em Proc. Amer. Math. Soc.}, 116:279--284, 1992.

\bibitem{garay1}
B.~M. Garay.
\newblock Strong cellularity and global asymptotic stability.
\newblock {\em Fund. Math.}, 138:147--154, 1991.

\bibitem{gobbino1}
M.~Gobbino.
\newblock Topological properties of attractors for dynamical systems.
\newblock {\em Topology}, 40:279--298, 2001.

\bibitem{graysonpugh1}
M.~Grayson and C.~Pugh.
\newblock Critical sets in 3-space.
\newblock {\em Int. Hautes \'{E}tudes Sci. Publ. Math.}, 77:5--61, 1993.

\bibitem{gunther1}
B.~G{\"u}nther.
\newblock A compactum that cannot be an attractor of a self-map on a manifold.
\newblock {\em Proc. Amer. Math. Soc.}, 120(2):653--655, 1994.

\bibitem{gunthersegal1}
B.~G{\"u}nther and J.~Segal.
\newblock Every attractor of a flow on a manifold has the shape of a finite
  polyhedron.
\newblock {\em Proc. Amer. Math. Soc.}, 119(1):321--329, 1993.

\bibitem{gunthersegal}
B.~G\"{u}nther and J.~Segal.
\newblock Every attractor of a flow on a manifold has the shape of a finite
  polyhedron.
\newblock {\em Proc. Amer. Math. Soc.}, 119:321--329, 1993.

\bibitem{hastings1}
H.~M. Hastings.
\newblock A higher-dimensional {P}oincar\'e-{B}endixson theorem.
\newblock {\em Glas. Mat. Ser. III}, 14 (34)(2):263--268, 1979.

\bibitem{hempel1}
J.~Hempel.
\newblock {\em 3-manifolds}.
\newblock Princeton University Press, 1976.

\bibitem{jaco1}
W.~Jaco.
\newblock {\em Lectures on three-manifold topology}, volume~43 of {\em CBMS
  Regional Conference Series in Mathematics}.
\newblock American Mathematical Society, Providence, R.I., 1980.

\bibitem{Jiang1}
B.~Jiang, Y.~Ni, and S.~Wang.
\newblock 3-manifolds that admit knotted solenoids as attractors.
\newblock {\em Trans. Amer. Math. Soc.}, 356(11):4371--4382, 2004.

\bibitem{jimenezllibre1}
V.~Jim{\'e}nez and J.~Llibre.
\newblock A topological characterization of the $\omega$--limit sets for
  analytic flows on the plane, the sphere and the projective plane.
\newblock {\em Adv. Math.}, 216:677--710, 2007.

\bibitem{peraltajimenez1}
V.~Jim{\'e}nez and D.~Peralta-Salas.
\newblock Global attractors of analytic plane flows.
\newblock {\em Ergodic Theory Dynam. Systems}, 29:967--981, 2009.

\bibitem{kaprod}
L.~Kapitanski and I.~Rodnianski.
\newblock Shape and {M}orse theory of attractors.
\newblock {\em Comm. Pure Appl. Math.}, 53(2):218--242, 2000.

\bibitem{kato1}
H.~Kato.
\newblock {Attractors in Euclidean spaces and shift maps on polyhedra}.
\newblock {\em Houston J. Math.}, 24:671--680, 1998.

\bibitem{magnuskarrasssolitar}
W.~Magnus, A.~Karrass, and D.~Solitar.
\newblock {\em Combinatorial group theory: {P}resentations of groups in terms
  of generators and relations}.
\newblock Interscience Publishers, 1966.

\bibitem{moise2}
E.~E. Moise.
\newblock {\em Geometric topology in dimensions 2 and 3}.
\newblock Springer-Verlag, 1977.

\bibitem{nortonpugh1}
A.~Norton and C.~Pugh.
\newblock Critical sets in the plane.
\newblock {\em Michigan Math. J.}, 38(3):441--459, 1991.

\bibitem{ortegayo1}
R.~Ortega and J.~J. S{\'a}nchez-Gabites.
\newblock A homotopical property of attractors.
\newblock {\em Topol. Methods Nonlinear Anal.}, 46:1089--1106, 2015.

\bibitem{pacoyo1}
F.~{R. Ruiz del Portal} and J.~J. S{\'a}nchez-Gabites.
\newblock \v{C}ech cohomology of attractors of discrete dynamical systems.
\newblock {\em J. Diff. Eq.}, 257(8):2826--2845, 2014.

\bibitem{rolfsen1}
D.~Rolfsen.
\newblock {\em Knots and Links}.
\newblock AMS Chelsea Publishing, 2003.

\bibitem{mio5}
J.~J. S\'anchez-Gabites.
\newblock How strange can an attractor for a dynamical system in a
  $3$--manifold look?
\newblock {\em Nonlinear Anal.}, 74:6162--6185, 2011.

\bibitem{mio6}
J.~J. S{\'a}nchez-Gabites.
\newblock Arcs, balls and spheres that cannot be attractors in $\mathbb{R}^3$.
\newblock {\em Trans. Amer. Math. Soc.}, 368(5):3591--3627, 2016.

\bibitem{mio7}
J.~J. S\'anchez-Gabites.
\newblock On the set of wild points of attracting surfaces in $\mathbb{R}^3$.
\newblock {\em Adv. Math.}, 315:246--284, 2017.

\bibitem{sanjurjo1}
J.~M.~R. Sanjurjo.
\newblock Multihomotopy, \v{C}ech spaces of loops and shape groups.
\newblock {\em Proc. London Math. Soc. (3)}, 69(2):330--344, 1994.

\bibitem{sanjurjo4}
J.~M.~R. Sanjurjo.
\newblock On the structure of uniform attractors.
\newblock {\em J. Math. Anal. Appl.}, 192(2):519--528, 1995.

\bibitem{sanjurjo}
J.M.R. Sanjurjo.
\newblock Multihomotopy, \v{C}ech spaces of loops and shape groups.
\newblock {\em Proc. London Math. Soc. (3)}, 69:330--344, 1994.

\bibitem{Souto}
J.~Souto.
\newblock A remark about critical sets in $\mathbb{R}^3$.
\newblock {\em Rev. Mat. Iberoam.}, 35(2):461--469, 2019.

\bibitem{Swarup}
G.~A. Swarup.
\newblock Finding incompressible surfaces in {$3$}-manifolds.
\newblock {\em J. London Math. Soc. (2)}, 6:441--452, 1973.

\end{thebibliography}

\end{document}